\documentclass[12pt]{article}
\usepackage{amssymb,amsmath,amsthm}
\usepackage[dvips]{graphics}

\newtheorem{theorem}{Theorem}
\newtheorem{lemma}[theorem]{Lemma}
\newtheorem*{implications}{Implications}
\newtheorem{proposition}[theorem]{Proposition}
\newtheorem{corollary}[theorem]{Corollary}
\newtheorem*{putative*}{Putative Theorem}
\newtheorem{question}{Question}
\newtheorem{conjecture}[question]{Conjecture}

\newcommand\MJWPsi{F}
\newcommand\MJWPhi{G}

\newcommand\eps{\varepsilon}

\begin{document}

\title{Lion and Man -- Can Both Win? }  \author{
  B. Bollob\'as$^\dag$\thanks{University of Memphis, Department of
    Mathematics, Dunn Hall, 3725 Noriswood, Memphis, TN 38152, USA.}
  \and I. Leader\thanks{Department of Pure Mathematics and
    Mathematical Statistics, University of Cambridge, Cambridge CB3
    0WB, UK} \and M. Walters\thanks{Queen Mary University of London,
    London E1 4NS.}  } \maketitle
\begin{abstract}
This paper is concerned with continuous-time pursuit and evasion games. 
Typically, we
have a lion and a man in a metric space: they have the same speed, and
the lion wishes to catch the man while the man tries to evade capture.
We are interested in questions of the following form: is it the case that
exactly one of the man and the lion has a winning strategy?

As we shall see, in a compact metric space at least one of the players has a 
winning strategy. We show that, perhaps surprisingly, there are examples 
in which both players have
winning strategies. We also construct a metric space in which, 
for the game with
two lions versus one man, neither player has a winning strategy. We prove
various other (positive and negative) related results, and pose some
open problems. 

\end{abstract}

\section{Introduction}

Rado's famous `Lion and Man' problem (see [8, pp.~114-117]
or [2, pp.~45-47]) is as follows. A lion and a man (each
viewed as a single point) in a closed disc have equal maximum speeds; 
can the lion catch the
man? This has been a well known problem since at least the 1930s -- 
it was popularised extensively by Rado and subsequently by Littlewood.
The reader not familiar with this problem is urged to give it a
few minutes' thought before proceeding further.

For the `curve of pursuit' (the lion always running directly towards
the man) the lion gets arbitrarily close to the man but does not ever
catch him. However, the apparent `answer' to the problem is that the
lion can win by adopting a different strategy, namely that of staying
on the same radius as the man. In other words, the lion moves, at top
speed, in such a way that he always lies
on the radius vector from the centre to the man. If we assume, as
seems `without loss of generality', that the man stays on the boundary
of the circle, then it is easy to check that the lion does now catch
the man in finite time.  Indeed, in the time that it
takes the man to run a quarter-circle
at full speed the lion (say starting at the centre) performs a
semicircle of half the radius of the disc, thus catching the man.

This `answer' was well known, but in 1952 Besicovitch (see 
[8, pp.~114-117]) showed that it is wrong to assume that the man
should stay on the boundary, and that in fact the man can survive
forever. His beautiful argument goes as follows.  We split time into a
sequence of intervals, of lengths $t_1,t_2,t_3,\ldots$. At the $i$th
step the man runs for time $t_i$ in a straight line that is
perpendicular to his radius vector at the start of the step. He
chooses to run into the half plane that does not contain the lion (if
the lion is on the radius then either direction will do). So certainly
the lion does not catch the man in this time step. The man then
repeats this procedure for the next time step, and so on.
%
%

Now, note that if $r_i$ is the distance of the man from the origin at the start
of the $i$th time step then $r_{i+1}^2=r_i^2+t_i^2$. Hence as long as
$\sum_i t_i$ is infinite then the man is never caught, and if
$\sum_i t_i^2$ is finite then the $r_i$ are bounded, so that (multiplying by
a constant if necessary) the man does not leave the arena. So taking for
example $t_i=1/i$ we have a winning strategy for the man.

We pause for a moment to reassure the reader that all of these terms like
`winning strategy' have a precise definition, and indeed there is really 
only one
natural choice for the definitions. Thus  
a {\it lion path} is a function $l$ from
$[0,\infty)$ to the closed unit disc $D$ such that $\vert l(s)-l(t) \vert
  \leq \vert s-t \vert$ for all $s$ and $t$ (in other words, the path
  is `Lipschitz' -- this corresponds to the lion having maximum speed
  say 1) and with $l(0)=x_0$ for some fixed $x_0$ in the disc (as, for
  definiteness, a starting point should be specified). A {\it man path}
is defined similarly. 
We write $L$ for the set of lion paths and $M$ for
  the set of man paths. Then a {\it strategy} for the lion is a
  function $\MJWPhi$ from $M$ to $L$ such that if $m,m' \in M$ agree on
  $[0,t]$ then also $\MJWPhi(m)$ and $\MJWPhi(m')$ agree on $[0,t]$. This
  `no lookahead' rule is saying that $\MJWPhi(m)(t)$
  depends only on the values of $m(s)$ for $0 \leq s \leq t$ (or
  equivalently, by continuity of $m$, that it depends only on the
  values of $m(s)$ for $0 \leq s < t$).  A strategy $G$ for the lion 
is a {\it winning
    strategy} if for every $m \in M$ there is a time $t$ with
  $\MJWPhi(m)(t)=m(t)$. We make corresponding definitions for the man.
Note
  that all of these definitions also make sense with the disc replaced by
  an arbitrary metric space $X$. We will usually suppress the dependence on
the starting points, speaking for example about just `the game on $X$', 
since for most results the actual starting points are
irrelevant (as long as they are distinct, of course).

Let us briefly remark that it would be quite wrong to use a different 
definition of `lion strategy' (for example) by insisting on some kind of 
finite delay, so
that the lion's position at time $t$ depended on the man's position at times
up to $t-\eps$ or some such. For then one would be disallowing such natural
strategies as `aim for the man' or `keep on the same radius as the man', and
thus one would be fundamentally changing the nature of the problem. (The 
relationship between strategies and finite delays 
will in fact be considered in some detail in Section 2.)

We now ask the question that motivates the work of this paper. We have
seen that the man has a winning strategy (the Besicovitch strategy); could
it be that the lion also has a winning strategy?

At first sight, this seems an absurd question to ask: after all, if both
players have winning strategies then let us consider a play of the game in
which each is following his winning strategy, and ask who wins? But a
moment's careful thought reveals that this `proof' is in fact nonsense.
For what would it mean to have a play of the game in which `each player
followed his strategy'? If the lion is using strategy $\MJWPhi$ and the man
is using strategy $\MJWPsi$, then we would need paths $l \in L$ and $m \in M$
such that $l=\MJWPhi(m)$ and $m=\MJWPsi(l)$. And there is no {\it a priori}
reason why such a common fixed point should exist. (There is also no reason
why such a common fixed point, if it exists, should be unique, but that is
a different aspect.)

Now, it turns out that, in this particular case, the `local finiteness' of
the Besicovitch strategy means that it is quite simple to show that the lion
cannot also have a winning strategy (see Section 2). But what would happen in
a different space (in other words, with the closed disc replaced by a
different metric space)? Indeed, in a different space, why should it even be 
true
that at least one of the man and the lion has a winning strategy?

One might imagine that this is merely some `formal nonsense', and that, once
thought about from the correct viewpoint, it would become clear that exactly
one of the lion and the man has a winning strategy. Surprisingly, this
is not the case.

The plan of the paper is as follows. We start in Section 2 by considering the
bounded-time version of the lion and man game (the man wins if he can stay
alive until some fixed time $T$). In this case, if we considered what
one might call the `discrete' version of the game, in which the two players
take turns to move, each move being a path lasting for time $\epsilon$
(for a fixed $\epsilon >0$), then we would be
in the world of finite-length games, and here of course no pathology
can occur. So it is natural to seek to approximate the continuous game
by the discrete version. Using this approach, we are able to show
that, in a compact metric space, at least one player has a winning
strategy for the bounded-time game. (Curiously, at one point the
argument seems to make essential use of the Axiom of Choice.) 
We do not
see how to extend our result to the original unbounded-time problem.

The methods of Section 2 seem to come very close to proving that it is
also true that at most one player can have a winning strategy. Indeed,
as we shall see, it seems that one is only an `obvious' technical lemma
away from proving this. But it turns out that this technical lemma is
not true. And in fact in Section 3 we present some
examples of metric spaces (even compact ones) in which both players
have winning strategies -- in the strongest possible sense, namely that
the lion can guarantee to catch the man by a fixed time $T$ and the man
can guarantee to stay alive forever. Interestingly, the key here is to 
understand the nature of the required strategies; the spaces themselves are not
particularly pathological.
 
In Section 4 we consider a related game that we call `race to a
point'. Two players, with equal top speeds, start at given points in a
metric space, and race towards a given target point. The first to
arrive is the winner (with the game a draw if they reach the target at
the same time or if neither reaches the target). Of course, if the
metric space is compact then each player has a shortest path to the
target (as long as he has {\it some} path to the target of finite length, 
that is), so that either exactly one player has a winning strategy or
else both players have drawing strategies. Thus the interest is in
noncompact spaces. One could view race to a point as a `simplified' 
version of lion and man, in the sense that the motion is towards a
fixed point (as opposed to towards or away from a moving point). 

We give an example of a space in which both players
have winning strategies for race to a point. We also give examples to show
that, perhaps more 
unexpectedly, there are spaces in which 
neither player has even a drawing strategy. We also
show how this game relates to the lion and man: based on our `race to
a point' examples we give a metric space in which, for the game of two
lions versus a man, neither side has a winning strategy.

In Section 5 we give a number of open problems.

Finally, in an appendix we prove some related results about `local 
finiteness':  we discuss
more general locally finite phenomena (like the Besicovitch strategy), as
well as giving applications of this to other games such as 
`porter and student', 
where one player seeks to leave a  region via a specified boundary and the
other player wishes to catch him the instant he reaches this boundary.
\bigskip

There has been a considerable amount of work on the lion and man
problem and related questions. 
For example, Croft~\cite{Croft} showed that if the man's path
is forced to have uniformly bounded curvature then the lion can catch
the man (although, strangely, the `stay on the same radius' strategy
does not achieve this).  Croft also showed that in the $n$-dimensional
Euclidean ball $n$ lions can catch the man while the man can escape
from $n-1$ lions. For some interesting versions played in a quadrant of the 
plane, see Sgall~\cite{Sgall}. There are also quantitative 
estimates about
how long it takes the lion to get within a certain distance of the
man: see for example \cite{Alonso}.

There is also a large body of work on `differential
games': see the beautiful book of Isaacs~\cite{Isaacs} for a thorough 
introduction, and Lewin~\cite{Lewin} for some applications to lion and man. For
results about (discrete) pursuit and evasion in general metric spaces, see
Mycielski~\cite{Mycielski}. However, none of the above appear to have
considered the particular questions we address here. 

Finally, it is
worth pointing out that our results do not seem to be related to
results about determinacy of infinite games in Descriptive Set Theory 
(see for example
Jech~\cite{Jech}), with the exception that our construction for the
`race to a point' game in Section 4 has perhaps some of the flavour of
some constructions of infinite games in which neither player has a
winning strategy.

\section{Finite Approximation}\label{s:finite_approximation}
We start by showing that, owing to the nature of the Besicovitch strategy,
there cannot be a winning strategy for the lion in the original 
lion and man game in the closed unit disc.

Let $\MJWPhi$ be a lion strategy and let $\MJWPsi$ be the (winning)
Besicovitch strategy defined earlier. We aim to play these two
strategies against each other: that is, to find a lion path $l$ and a man
path $m$ with $l=\MJWPhi(m)$ and $m=\MJWPsi(l)$. By definition of the
Besicovitch strategy, the man's path is determined for time $t_1$: we may write
this path as $m|_{[0,t_1]}$ (with slight abuse of notation, as we do not
yet know that $m$ exists). Now, given
$m|_{[0,t_1]}$, the lion's strategy determines $l|_{[0,t_1]}$, and, in
particular, determines $l(t_1)$. The Besicovitch strategy now tells the man 
what to do for
time $t_2$: that is, $m|_{[0,t_1+t_2]}$ is determined, so as before the
lion's path is determined up until $t_1+t_2$.

Repeating in this way we obtain the desired pair of paths: since the
man's strategy is winning we know that $l(t)\not=m(t)$ for all
$t$. That is, the lion does not win, so the lion's strategy is not a
winning strategy. 

Note that, in this proof, we used the special `discrete' nature of
the man's strategy: the man committed to doing something for some
positive amount of time. Many sensible strategies are not of this
form: indeed, in the strategy we gave earlier for the lion of `stay
on the same radius', the lion's position continually depends on where
the man is now (or, equivalently, where he was at all earlier times).

However, if one were to insist that strategies {\it should} be in some sense 
discrete,
or that there was some `delay' (the man's position at time $t$ being allowed 
to depend only on the
lion's position at times before $t-\eps$ for some fixed $\eps>0$, and
vice versa), then there should be no problem in proving that exactly one
player has a winning strategy. One might view this kind of restriction as
arising from a `real world' simplification of the problem.  

Based on this, let us try to approximate the real game by a discrete
version, as follows. Before we do so, we will introduce one change to
the game: in the {\it bounded-time lion and man game} in space $X$
there is a fixed parameter $T>0$, and the lion wins if he has caught
the man by time $T$ while the man wins otherwise.  For the rest of
this section, we shall be considering the bounded-time game.
 
Let $X$ be a metric space. Fix $\eps>0$ such that $T$ is an integer multiple 
of $\eps$: say
$T=n \eps$. In the {\it discrete bounded-time game} on $X$ the two
players take turns to move, with say the lion moving first. On a turn, the
player runs for time $\eps$ -- or, more precisely, he chooses a path of
path length at most $\eps$, starting at his current position, and runs along
it to its end.
The game ends after each player has had $n$ turns; the `outcome' of the game
is defined to be the closest distance $d$ that occurred between the lion and
man at any time. 

Since this is a finite game (a game lasting for a fixed finite number of
moves), it is easy to see (for example, by `backtracking')
that 
there is a $\delta=\delta(\eps)$ such that for any $\delta'<\delta$
the man has a strategy that ensures $d$ is at least $\delta'$ and for
any $\delta'>\delta$ the lion has a strategy that ensures $d<\delta'$.

[We remark in passing that this would remain the case if we were considering
a discrete version of the unbounded-time game, by virtue of the fact
that Borel games are determined (see e.g.\ Jech \cite{Jech}). However, 
it turns out that
this does not seem to help with the analysis of the unbounded-time game.]

As we are interested in the relation of this discrete game to the
original `continuous' game, it is natural to consider $\eps\to
0$ (of course, only through values that divide $T$). One would hope that 
$\delta\to 0$ as $\eps\to 0$ corresponds to a
lion win in the continuous game. More precisely, one would hope 
that the following four implications hold.
\begin{implications} 
  For the bounded-time game on the metric space $X$, with $\delta$ and 
$\eps$ as above, we have
\begin{enumerate} 
  \item\label{ldc} If $\delta\to 0$ as $\eps\to0$ in the discrete game
    then the lion wins the continuous game.
  \item\label{mdc} If $\delta\not\to 0$ as $\eps\to0$ in the 
discrete game then the
    man wins the continuous game.
  \item\label{lcd} If the lion wins the continuous game then
    $\delta\to0$ as $\eps\to0$ in the discrete game.
  \item\label{mcd} If the man wins the continuous game then
    $\delta\not\to0$ as $\eps\to0$ in the discrete game.
\end{enumerate}
\end{implications}
Now, as we shall see, Implications
\ref{mdc} and \ref{lcd} are trivial. Implication~\ref{ldc} is true, if $X$ is
compact, but
seems to be not quite trivial; indeed, our proof needs the Axiom of
Choice. Combining Implications~\ref{ldc} and~\ref{mdc} we see that at
least one player has a winning strategy for the bounded-time
game played in a compact space.

Under some extra, seemingly mild, conditions the final implication ~\ref{mcd}
is easy to prove. Surprisingly, however, it is false in general, as we see
in Section 3.

We remark that for our purposes it does not matter who moves first
in the discrete-time game. Indeed, allowing the man to move first, or
equivalently forcing the lion to stay where he is for his first move, will
only change the value of $\delta$ by at most $\eps$, and so will not change
whether or not $\delta\to0$ as $\eps\to0$.

Note that a strategy for a player in the $\eps$-discrete game (with that player
moving first) gives rise naturally to a strategy for that player in the
continuous game. Conversely, a strategy for a player in the continuous game
gives rise naturally to a strategy for that player in the $\eps$-discrete
game (with that player moving second). Here the restrictions about who
moves first are to ensure that the no-lookahead rule is not violated.
Note also that these changes from discrete to continuous or vice versa change
the closest distance between lion and man (in any play of the game) by at
most $\eps$.

In Lemmas~\ref{l:mdc}--\ref{l:aa} and Corollary~\ref{c:ldc} we fix a
metric space $X$ and we consider only bounded-time games.

\begin{lemma}\label{l:mdc}
  Suppose that
  $\delta\not\to0$ as $\eps\to0$. Then the man has a winning strategy
  in the continuous game.
\end{lemma}
\begin{proof}
Choose $\eps >0$ such that $\delta(\eps)>\eps$, and let $\MJWPsi$ be a man
strategy for the $\eps$-discrete game (with the man moving first) 
witnessing this.
Then, in the continuous game, the man just follows 
the corresponding strategy. At all times, the 
lion is at most $\eps$ closer to the man than he is in the
discrete game, and so does not catch the 
man.
\end{proof}

\begin{lemma}
Suppose that the lion has a winning strategy for the continuous
game. Then $\delta\to 0$ as $\eps\to 0$.
\end{lemma}
\begin{proof}
  Let $\MJWPhi$ be a winning strategy for the lion in the continuous
game, and let $\MJWPhi'$ be the corresponding strategy for the lion in the
$\eps$-discrete game (with the lion moving second). In any play of the discrete
game, the lion (following $\MJWPhi'$) must be at distance at most
$\eps$ of the man (because the lion catches the man in the continuous game).
Hence $\delta(\eps)\le \eps$ and the result follows.
\end{proof}

The next lemma assumes that $X$ is compact; this is a natural condition to 
impose for
the whole lion-man game in general, and here it is important because it 
allows us to
take limits of paths. We remark that, since all paths are Lipschitz, the spaces
$L$ and $M$, viewed as metric spaces with the supremum metric, are compact.
This is by a standard Arzel\` a-Ascoli type 
argument (see for example [3, Ch.~6])
-- the fact that all paths are Lipschitz guarantees equicontinuity.

To prove that a winning strategy for the 
lion in the discrete game
lifts to a winning strategy in the continuous game we use the
following lemma.

\begin{lemma}\label{l:aa}
Let $X$ be compact. Suppose that for every $n$ there exists a lion 
strategy
$\MJWPhi_n$ in the continuous-time game such that for every man path 
$m$ we have
$d(\MJWPhi_n(m)(t),m(t))<1/n$ for some $t$. 
Then there exists a winning lion strategy in the
continuous-time game.
\end{lemma}
\begin{proof}
It is tempting to argue as follows. For any man path $m$, the compactness of
$L$ ensures that there exists a subsequence of the paths $\MJWPhi_n(m)$
that converges uniformly to some path $l$. Define $\MJWPhi(m)=l$, noting
that since for each $n$ there is a $t$ with $\MJWPhi_n(m)(t)$ within 
distance $1/n$ of $m(t)$ it follows by the uniformity of the convergence
(and the fact that the interval $[0,T]$ is compact) that we have
$\MJWPhi(m)(t) = m(t)$ for some $t$.
However, this may not yield a valid strategy: there is no reason why
$\MJWPhi$ should satisfy the no-lookahead rule.

Instead of this, we build up $\MJWPhi$ one path at a time -- or,
to put it another way, we use Zorn's Lemma to construct $\MJWPhi$.
Let a {\it partial strategy} be a
function $\MJWPhi$ from a subset of $M$ to $L$ that satisfies `no lookahead'
where it is defined -- in other words, if $\MJWPhi$ is defined at
$m,m' \in M$, and $m$ and $m'$ agree on 
$[0,t]$, then also $\MJWPhi(m)$ and $\MJWPhi(m')$ agree on $[0,t]$. We say that
a partial strategy $\MJWPhi$ is {\it good} if
for each $m$ for
which $\MJWPhi(m)$ is defined there 
is a subsequence of the paths $\MJWPhi_n(m)$
that converges uniformly to $\MJWPhi(m)$. 
Given two good partial 
strategies $\MJWPhi_1$ and $\MJWPhi_2$ with domains $M_1$ and $M_2$
respectively, we say $\MJWPhi_1 \leq \MJWPhi_2$ if $M_1\subset M_2$ and
$\MJWPhi_2|M_1=\MJWPhi_1$. 

It is obvious that every chain of good partial strategies 
has an upper bound, namely their union. Hence by Zorn's Lemma there is
a maximal good partial strategy $\MJWPhi$, say with domain $M'$. We will show 
that 
this is a full strategy, i.e. that $M'=M$.

Indeed, suppose not. Fix $m\not \in M'$. We aim to extend $\MJWPhi$ to a
good partial strategy $\MJWPhi'$ on $M'\cup \{m\}$. Set $t_0$ be 
\[ t_0 = 
\sup \{t:\exists
  m'\in M',\ \forall s<t,\ m(s)=m'(s)\}.
\]

There are now two cases, according to whether or not this supremum
is attained. If it is attained, we have a path $m'$ in the domain of
$\MJWPhi$ agreeing with
$m$ on $[0,t]$. Then a certain subsequence of the $\MJWPhi_n(m')$, 
say the $\MJWPhi_{n_i}(m')$, converges to
$\MJWPhi(m')$. We may now choose a convergent subsequence of the 
$\MJWPhi_{n_i}(m)$,
and let $\MJWPhi'(m)$ be the limit of that sequence.

On the other hand, if the supremum is not attained then we have a sequence
$t_1,t_2,\ldots$ tending up to $t_0$, and paths $m_1,m_2,\ldots$ in the domain
of $\MJWPhi$, such that $m$ and $m_i$ agree on $[0,t_i]$. For each $i$, we may
choose $n_i$ such that $\MJWPhi_{n_i}(m_i)$ is within distance $1/i$ of
$\MJWPhi(m_i)$ (and say $n_1<n_2<\ldots$). We now choose a convergent
subsequence of the $\MJWPhi_{n_i}(m)$,
and let $\MJWPhi'(m)$ be the limit of that sequence.

It is easy to check that $G'$ is indeed a good partial strategy.
\end{proof}

We remark that Lemma~\ref{l:aa} may also be proved using limits along a 
(non-principal) ultrafilter.
But it would be interesting to know if the appeal to Zorn's Lemma (or
similar) is really necessary.
 
By taking as the $G_n$ the strategies corresponding to the lion strategies in
the discrete game (with the lion moving first), we 
immediately have the following corollary.

\begin{corollary}\label{c:ldc} Suppose that in the $\eps$-discrete 
game we have $\delta(\eps)\to0$ 
as $\eps\to 0$. Then the lion has a winning strategy for the 
continuous-time game. \hfill\qed
\end{corollary}
Combining Lemma~\ref{l:mdc} and Corollary~\ref{c:ldc} we have:
\begin{theorem}
  In the bounded-time game played in a compact metric space, 
 at least one of the lion and man has a
  winning strategy.\hfill\qed
\end{theorem}
The condition that the game be played for bounded time seems crucial for the
above argument: we
do not know what happens if time is unbounded. We remark that there
are simple spaces showing that the man being able to escape for an 
arbitrarily long time does not mean that the man can escape forever. Indeed, 
consider a space consisting of paths from a point of path-lengths
$1,2,3,\ldots$, with the man starting at the common point. Then the man has
a winning strategy for the bounded-time game (for any $T$), but not for the
unbounded-time game. Note that this space can easily be made compact, by
`rolling up' the paths. 

We also do not know what happens if the metric space $X$ is not compact.
We suspect that it can happen that (even for the bounded-time game)
neither player has a winning strategy, but we have been unable to show this.
However, in Section 4 we will show that if one allows two lions (acting as a
team) to pursue a man then it can indeed happen that neither side has a 
winning strategy.

We now turn to the last of our four implications. If the man has a winning 
strategy that is
continuous (as a function from $L$ to $M$) then the result is immediate.

\begin{lemma}
Let $X$ be a compact metric space. Suppose that the man has a 
continuous strategy that is winning for the
bounded-time game on $X$. Then $\delta\not\to0$ as $\eps\to0$.
\end{lemma}
\begin{proof}
Let $\MJWPsi$ be the continuous winning strategy.  The function mapping a
lion path $l$ to $\inf_{t\in[0,T]}d(l(t),\MJWPsi(l)(t))$ (in other words,  
the closest the lion ever gets to the man when the
man plays this strategy), is
continuous as a function of $l$. As $L$ is compact, 
there is a path minimising this distance. But $\MJWPsi$ is a winning strategy, 
and so this minimal distance must be strictly greater than zero -- say it
is $c>0$.
Hence if the man plays the corresponding strategy in the $\eps$-discrete 
game (with the man moving second) then for 
any $\eps<c$ we have $\delta(\eps)\ge c-\eps$, so that $\delta$ does 
not tend to
zero.
\end{proof}

Now, do strategies tend to be continuous? As defined earlier, the
Besicovitch strategy for the game in the closed disc is {\it not} continuous.
This is for a reason that one feels ought to be easy to get round: that 
if the lion is on the same radius as the man
then the man makes an arbitrary choice of which way to run.  
But in fact there is {\it no} continuous winning strategy for the man in 
this game. 
\begin{theorem}
In the lion and man game in the closed unit disc the man does not have
a continuous winning strategy. Indeed, for any continuous man strategy
there is a lion path catching the man by time $1$.
\end{theorem}
\begin{proof}
Suppose that $\MJWPsi$ is a continuous man strategy.  For each point $z$
in the unit disc define the lion path $l_z$ to be the constant speed
path from the origin to $z$ reaching $z$ at time 1. Then $z\mapsto
\MJWPsi(l_z)(1)$ is a continuous function of $z$ from the disc to the
disc. Hence by the Brouwer fixed point theorem there is some $z$ with
$\MJWPsi(l_z)(1)=z=l_z(1)$. In other words, the man is caught by the lion.
\end{proof}

One might still feel that the problems here are merely technical,
arising as they do out of the arbitrary choice when the lion is on the
same radius as the man. It would be natural to imagine that this can be got 
round by allowing multivalued strategies (thus each lion path would map to
a set of man paths) -- so if the lion was on the same radius as the man then
we would have two paths extending our path so far, and so on. One would then 
hope to prove that such a set-valued function may always be chosen to be
upper semi-continuous (or have some related property), which would then
allow a proof as in the lemma above.

But, contrary to what the authors of this paper had thought for some time, 
this cannot
be made to work, and indeed the whole theorem that exactly one player has a
winning strategy is false -- even in compact spaces. This is the content
of the next section.

\section{Some Examples}\label{s:2wins}

We start by giving an example of a compact metric space in which, for the
unbounded-time game (and also for the bounded-time game), 
both players have winning strategies. As we shall see,
the strategy for the lion is very simple. The strategy for the man, on the
other hand, will 
rely on a curious device of `getting out from underneath the lion'.

If $X$ and $Y$ are metric spaces then the {\it $l_\infty$ sum} 
of $X$ and $Y$ is
the metric space on $X \times Y$ 
in which the distance from $(x,y)$ to $(x',y')$ is
$\max (d(x,x'),d(y,y'))$. 

\begin{theorem}
Let $X$ be the $l_\infty$ sum of the closed unit disc $D$ and $[0,1]$. 
Then, in the lion and
man game on $X$ with the man starting at
$(0,0)$ and the lion starting at $(0,1)$, both players have
winning strategies.
\end{theorem}
\begin{proof}
The lion has an obvious winning strategy: keep the `disc' coordinate
the same as the man and run towards him in the interval coordinate. Thus the
lion catches the man in time at most 1.

For a winning strategy for the man, the first aim is to `escape from
underneath the lion'. For time $1/2$ (say), the man acts as follows.
If there exists a positive time $t$ such that the lion's disc
coordinate was exactly $s$ for all $0 \leq s \leq t$ then the man
runs straight to $(-1/2,0)$, while if there is no such positive $t$
then the man runs straight to $(1/2,0)$. Note that this satisfies the
no-lookahead rule, because for any given time $t \leq 1/2$ the man's
position at time $t$ is determined by the lion's position at arbitrarily
early times.

At time $1/2$, the man now has a different disc coordinate to the lion.
He then plays the Besicovitch strategy in the disc (and does whatever he
likes in the other coordinate).
\end{proof}

The above example does not embed isometrically into Euclidean space, because 
the sum is taken as an $l_\infty$ sum -- and 
moreover the $l_\infty$ nature of the sum was crucial to the lion having a
winning strategy. It would be very interesting to know if such a
construction exists in Euclidean space.

In the above example, the start positions were very important. Indeed,
if the lion and man did not start with the same disc coordinate then
the lion does not have 
a winning strategy -- this is as for the
earlier discussion of the Besicovitch strategy. However, we now 
show that by adapting
the above ideas, and getting the man to use the `escape from underneath the
lion' idea not just once but repeatedly, there is a compact metric space in 
which, for any (distinct) starting positions, both players have
winning strategies.
\begin{theorem}
  Let $X$ be the closed unit ball in
  $l_\infty^2$ (the $l_\infty$ sum of $[-1,1]$ with itself). Then both players 
have winning strategies for the lion and man game on $X$, 
for any distinct starting 
positions.
\end{theorem}
\begin{proof}
We describe the strategy when the man starts at $(1,0)$ and the lion starts 
at the point
$(x_0,y_0)$ -- the general case is the same. The lion has an obvious winning 
strategy: in each cordinate, run at full speed towards the man, and
when that coordinate is equal to that of the man then keep the coordinate the
same as that of the man. This catches the man in time at most $2$.

For the man's strategy, we will give a strategy for the man which starts with  
him running at full
speed to one of $(0,1)$ and $(0,-1)$ without being caught. From here he 
repeats the strategy, thus surviving for all time.

We split into three cases: if $y_0<0$ then the man runs to $(0,1)$ and
so does not get caught. If $y_0>0$ then the man runs to $(0,-1)$ and, again, 
does not
get caught.

Finally we deal with the case $y_0=0$. Here we follow the strategy from the
previous proof: if the lion's $y$-coordinate at time $s$ is equal to $s$ for
all $0<s<t$, some $t>0$, then the man runs to $(0,-1)$, and otherwise he
runs to $(0,1)$. Again, he is not caught.
\end{proof}

\section{Race to a Point}\label{s:race_to_a_point}

In this section we give examples of race to a point games in which
both players have a winning strategy, and in which
neither player has even a drawing strategy (meaning, of course, a 
strategy that guarantees the player either a win or a draw in each play of
the game). The first of these is included
for the sake of completeness, and also out of interest, because what is 
somehow the `obvious' example is in fact not an example at all. The second
will form the basis for an example of a lion and man game 
involving two lions pursuing a man.

In general, if the game is symmetric, with the two players starting at
the same point, then it is clear that each player has a drawing
strategy, namely `copy the other player'. What about both players
having a winning strategy?

It is natural to try to construct an example based on the idea that, in a game
where one has to name a higher number than one's opponent, saying
`my opponent's number plus 1' would be a sensible thing to do, if it were
allowed. So we let $X$ consist of two points $x$ and $y$ at distance 1,
joined by some disjoint paths of length $1+1/n$ for every $n$ (the distance 
apart of two points on different paths is irrelevant to the argument). 
Both players 
start at $x$ and the target is $y$. Then the analogue of the above would be
for Player $A$ to move as follows: if Player $B$ runs along the path of
length $1+1/n$, then Player $A$ runs along the path of length $1+1/(n+1)$ 
instead, thus 
arriving at $y$ before $B$ does.

Surprisingly, this is not correct. For unfortunately 
Player $B$ might `backtrack' and change his path. Or he might
wait at $x$ for a while before proceeding (and there may not even be a
`first' path that he moves onto). So the above strategy is not 
well-defined. In fact, there is no winning strategy, as we now show. This 
proposition should be contrasted with the one that follows it.

\begin{proposition}
  Let $X$ be the metric space defined above. Then in the race to a point
  game starting at $x$ and ending at $y$ neither
  player has a winning strategy.
\end{proposition}
\begin{proof}
Let $\MJWPhi$ be a strategy for Player A. Suppose Player $B$'s path $p$ is just
constant at the start point. Then, according to  $\MJWPhi$, Player $A$'s path 
$q=\MJWPhi(p)$ reaches the finish
at some time. On this path, at some time $t$ Player $A$ is strictly more
than distance 1 from the finish: say he is at point $z$, at distance $1+\eps$.

Now, by the no-lookahead rule it follows that for every Player $B$ path that
stays at the start point until time $t$ we have that Player $A$ is at the same
point $z$ (at distance $1+\eps$ from the finish) at time $t$. 
But now consider a
path for Player $B$ that waits at the start until time $t$ and then runs
straight to the finish along a path of length $1+1/n < 1+\eps$. In that play
of the game, $B$ reaches the finish before $A$.
\end{proof}

However, we can modify the idea of this construction, by allowing the
various paths from $x$ to $y$ to connect to each other. The simplest way
to do this is as follows.

\begin{proposition}
  Let $X$ be the subset of the plane consisting of the open upper half plane
  with $(0,0)$ and $(1,0)$ added. Then in the race to a point game on $X$
  starting at $(1,0)$ and finishing at $(0,0)$ both players have a
  winning strategy.
\end{proposition}

\begin{proof} We give a winning strategy for (say) Player A. 
It is convenient to work in polar coordinates $(r,\theta)$.  
Suppose that Player B
follows path $(r(t),\theta(t))$. Then Player 1 follows the path
$(s(t),\phi(t))$ given by
\begin{align*}
s(t)&=(r(t)+2(1-t))/3\\
\phi(t)&=t+s(t)-1\; .
\end{align*}
The key point is that $s(t)$ is always smaller than $r(t)$ for $t>0$ (because
we cannot ever have $r(t)=1-t$). The choice of $\phi$ ensures that
the path is Lipschitz and at the same time stays within the space.
\end{proof}

We now turn to our main aim in this section: a metric space (with
given start points for the two players and a given finish point) in
which neither player has even a drawing strategy.

\begin{theorem}
  There is a space $X$ in which for the race to a point game (with specified
starting positions and target) neither player has even a drawing strategy.
\end{theorem}
\begin{proof}
The space will be the following: we will pick subsets $A$ and $B$ of the
interval
$(0,1)$. The space $X$ will be a subset of the complex plane: it will
consist of the subset $\{ e^{it}:\; 0 \leq t \leq 1 \} \cup 
 \{ -e^{it}:\; 0 \leq t \leq 1 \}$ of the unit circle, 
together with `spokes' (radii to the origin)
from each point $e^{ia}$, $a \in A$ and $-e^{ib}$, $b \in B$.

Player~A starts at $1$ and Player~B at $-1$, and they are
racing to the origin. We start by finding conditions on $A$ and $B$ 
that would ensure that neither player has
a drawing strategy.

Suppose that Player~A has a drawing strategy.  Fix a point $b$ in $B$
with $b>1/2$. Consider the Player~B path going round the circle to $-e^{ib}$
and then along the spoke to the origin. Player~A's strategy gives some
path $p$, which of course must leave the unit circle at some time
(as it has to reach the origin). Let $t_0$
be the last time at which Player~A is on the boundary circle -- certainly
$t_0 \leq b$. Since Player~A is following a strategy he
follows exactly the same path $p$ whenever Player~B sets off round the
circle at full speed, at least until Player~B deviates from the path
above. If there is some time $t$ at which Player~B is at the end of
one of his spokes and Player~A is not on or at the end of one of his,
then the Player~B path `go around the boundary until $t$ and then
along the spoke' beats Player~A's strategy (since at the time Player~B
leaves the circle Player~A is not able to).

Now, during the time $[0,t_0]$ Player~A's position varies
continuously. Thus, if no such $t$ occurs, then there is a continuous
function (even Lipschitz) from $[0,t_0]\to[0,t_0]$ mapping 0 to 0 such
that every point of $B$ is mapped to a point of $A$.

The same of course has to hold with the players reversed. Thus we shall
be done if we can find sets $A$ and $B$ in $(0,1)$ such that for no 
$t_0>0$ is there a 
continuous function $f:[0,t_0]\to[0,t_0]$ with $f(0)=0$ and
$f(B)\subset A$, and nor is there such a continuous function with
$f(A) \subset B$.

We use a well-ordering argument to construct these sets. It
will be convenient to use the term {\em large}  to mean of the same
cardinality as the reals and the term {\em small} to mean of strictly
smaller cardinality than the reals. The reader is encouraged to think
of these as uncountable and countable respectively (although formally that
would assume the continuum hypothesis).

Consider the set $S_n$ consisting of all continuous functions from
$[0,1/n]$ to itself that fix zero, for each natural number $n$, and let 
$S$ denote their union (over all $n$). Then $S$ has the same cardinality
as the reals, and so we may well-order $S$ as 
$\{ f_{\beta}:\ \beta<\alpha \}$, where $\alpha$ is the first  
ordinal with cardinality that of the reals.

We construct the sets $A$ and $B$
inductively: at any stage $\beta<\alpha$ we will have put 
at most a small number of 
points in $A$ and forbidden at most a small number of points from $A$ --- we 
will call these sets $A^+$ and $A^-$ respectively. And similarly for $B$ 
(we start with all of these sets empty). 
At the $\beta$th stage, consider the function $f=f_{\beta}$: say it maps
$[0,1/n]\to[0,1/n]$. We want to pick a point $b$ to put in $B$
and an $a$ to forbid from $A$ with $f(b)=a$: that is, we want to make
sure that $f(B)\not\subset A$. 

First, suppose $f([0,1/n])$ is large. Since there are only a small
number of points that we have put in $A$ so far (i.e., $A^+$ is small)
we have that $f^{-1}(f([0,1/n]\setminus A^+)$ is large. Since $B^-$ is also
small we can pick $b$ in $f^{-1}(f([0,1/n]\setminus A^+)\setminus
B^-$. We put $b$ in $B^+$ and $a=f(b)$ in $A^-$.

If $f([0,1/n])$ is not large, then, by the Intermediate Value Theorem,
it must be constantly zero. Since $0\not \in A$, we just pick a point
$b$ in $[0,1/n]\setminus B^-$ and place it in $B^+$.

We now do the same the other way round: putting a point $a'$ in $A^+$ and
$b'$ in $B^-$ with $f(a)=b'$.

Continuing in this way we obtain the sets we require by letting
$A=\bigcup A^+$ and $B=\bigcup B^+$. Indeed, suppose that we are given
a continuous function $f\colon[0,t_0]\to [0,t_0]$. Pick $n$ with
$1/n<t_0$. Then $f$ restricted to $[0,1/n]$ belongs to $S$, and so 
there is a point $a\not\in A$ and $b\in B\cap[0,1/n]$ with
$f(b)=a$.
\end{proof}

It would be interesting to know whether there is an explicit construction
(meaning without use of the Axiom of Choice or similar) of such an example.

We now show how to use the construction above to give a pursuit game of two 
lions against a man in which neither
player has a winning strategy.

The idea is to take the above space $X$  and attach an infinite ray
to the origin (out of the plane). Suppose first that we 
consider the usual lion and man game (with only one lion) on this
space, with the
lion and man starting at $(-1,0)$ and $(1,0)$ respectively. Then no lion
strategy always catches the man, since (as above) there is a man path that 
reaches
the origin before the corresponding lion path; if the man now runs at full 
speed along the infinite ray then he will not be caught. Unfortunately, it is
not correct to argue that the man cannot have a winning strategy merely 
because he cannot guarantee to reach the origin first -- it 
is easy to see that the man
can win while staying entirely within $X$.

However, the addition of a second lion is enough to render this impossible.
When we speak of `two lions' they are understood to form a team: the lions
win if at least one of them catches the man.

\begin{theorem}
  There is a metric space $Y$ in which, for the game of two lions against a
man (with specified starting positions), neither player has a winning
strategy.
\end{theorem}
\begin{proof}
We form $Y$ from $X$ by
attaching an infinite ray at the origin and also a line of length
1 from $(-2,0)$ to $(-1,0)$. The lions start at $(-2,0)$ and
$(-1,0)$ and the man starts at $(1,0)$.

As above we see that the no lion strategy stops the man getting to the
origin first (the second lion is too far away to affect this argument). Once on
the ray the man, of course, wins. Hence there is no winning strategy for the
lions.

To see that the man cannot have a winning strategy, note first that the
man can never guarantee to reach the infinite ray (thanks to the lion that
starts at $(-1,0)$). If we then let that lion stay at the origin for ever, we 
can let the other lion run around the circle, eventually trapping the man
on a spoke. Hence there is no winning strategy for the man.
\end{proof}

\section{Open Questions}

In this section we collect the various open problems that have been
mentioned in the paper, as well as giving some other ones.

The most interesting question is probably that of whether or not there exists
a metric space, necessarily non-compact, in which neither lion nor man has
a winning strategy in the bounded-time game.

\begin{question}
Does there exist a metric space $X$ in which for the bounded-time lion
and man game neither player has a winning strategy?
\end{question}

It would also be very nice to know what happens for the {\it
  unbounded-time} game in a compact (or indeed any) metric space.

\begin{question}
Let $X$ be a compact metric space. In the unbounded-time lion and man 
game must it be the case that at least one player has a winning strategy?
\end{question}

Then there is the question about whether or not the general situation
becomes nicer if we restrict to Euclidean spaces.

\begin{question}
Is there a subset of a Euclidean space for which both the lion and the
man have winning strategies for the bounded-time game?
\end{question}

We believe that the answer to this question is
no. Indeed, we believe that much stronger
statements should be true, determining exactly who wins in a given
subset of Euclidean space: the lion should win precisely when the subset
is `tree-like' or `dendrite-like'. An example of such a statement would be the 
following (where an
{\it arc} is an injective path).

\begin{conjecture}
Let $X$ be a subset of a Euclidean
space that has an associated path length metric (in other words, any two
points are joined by a path of finite length), with $X$ compact in this metric.
Then the lion has a winning 
strategy for the lion and man game on $X$
from all starting positions if and only if any two points of $X$ are
joined by a unique arc.
\end{conjecture}

Perhaps related to this is the following question. 
Lemma~\ref{l:aa} clearly applies to
any metric space $X$ that is isometric to the path length metric on a
compact metric space -- for example, it applies to the real line, since the
real line is the path length metric of a suitable compact subset of 
the plane. 

\begin{question}
Is there a natural larger class of spaces to which Lemma~\ref{l:aa} 
applies?
\end{question}

Finally, it would be interesting to know how much of our use of the Axiom of
Choice is actually necessary.

\begin{question}
  Is there a proof of Lemma~\ref{l:aa} that does not use Zorn's Lemma or
similar?
\end{question}

\begin{question}
  Is there an `explicit' (constructive) example of a metric space in
  which, for race to a point, neither player has a drawing strategy?
In particular, can this happen for a Borel subset of a Euclidean space?
\end{question}

\pagebreak
\section{Appendix}

In this section we generalize the argument given in
Section~\ref{s:finite_approximation} that the lion does not have a
winning strategy in the original lion and man game in the closed disc.
Recall that this was based on some particular properties of the
Besicovitch strategy.

For the lion and man game in a metric space $X$, we call a man
strategy $\MJWPsi$ {\em locally finite} if it satisfies the following
property: if  $l$ and $l'$ are any two lion paths that agree
on $[0,t]$ for some $t$ then the corresponding man paths $\MJWPsi(l)$
and $\MJWPsi(l')$ agree on $[0,t+\eps]$ for some $\eps>0$ (which may
depend on $l|_{[0,t]}$). Thus, informally, the man commits to doing
something for some positive amount of time dependent only on the
situation so far.

\begin{proposition}\label{p:lf}
  Suppose that $X$ is a complete metric space, that $\MJWPsi$ is a
  locally finite man strategy and that $\MJWPhi$ is any lion strategy.
  Then there exist paths $l\in L$ and $m\in M$ with
  $\MJWPhi(m)=l$ and $\MJWPsi(l)=m$. Moreover, these paths are unique.
\end{proposition}

The reader will see that the proof below would also apply with lion
and man interchanged. It follows that if one player has a locally
finite winning strategy then the other player does not have any winning
strategy.

We need one further definition: a {\em partial lion path} is the
restriction of any lion path to any interval $[0,t]$, and similarly
for a {\em partial man path}. If $l$ and $m$ are partial lion and
man paths defined on the same interval we say they are {\em
  compatible} (for given $\MJWPhi$ and $\MJWPsi$) if $\MJWPsi(l)=m$
and $\MJWPhi(m)=l$. (The notation $\MJWPhi(m)$ is, strictly speaking,
an abuse of notation, but the no lookahead rule ensures that its
meaning is clear.)

\begin{proof}

Suppose that $l,m$ are a compatible pair of partial paths on $[0,t_0]$
and $l',m'$ are a compatible pair on $[0,t_1]$ with $t_1>t_0$. We
claim that $l'|_{[0,t_0]}=l$ and $m'|_{[0,t_0]}=m$. Indeed, suppose
not. Then let
\[
s=\inf\{t:l(t)\not=l'(t)\text{ or }m(t)\not=m'(t)\}.
\]

Since $l(0)=l'(0)$ and $m(0)=m'(0)$ we see that $0\le s\le t_0$. Since
all the paths are continuous we have $l(s)=l'(s)$ and $m(s)=m'(s)$. In
particular, $s<t_0$.  Now, $\MJWPhi$ is locally finite, so there
exists some $\eps>0$ (chosen so that $s+\eps<t_0$) such that
$\MJWPhi(m')|_{[0,s+\eps]}$ depends only on $m'|_{[0,s]}$. Since
$m|_{[0,s]}=m'|_{[0,s]}$ we see that
$\MJWPhi(m')|_{[0,s+\eps]}=\MJWPhi(m)|_{[0,s+\eps]}$: in other words
$l'|_{[0,s+\eps]}=l|_{[0,s+\eps]}$. Finally, since $\MJWPsi$ is a
strategy we see that
$\MJWPsi(l')|_{[0,s+\eps]}=\MJWPsi(l)|_{[0,s+\eps]}$, that is
$m'|_{[0,s+\eps]}=m|_{[0,s+\eps]}$. This contradicts the definition of
$s$, thus establishing our claim.

Now consider the collection of all compatible pairs of partial paths. Let 
\[
s=\sup \text{ $\{t:\exists$ a compatible pair of partial paths $l,m$ on $[0,t]$\}}
\]
and suppose for a contradiction that $s$ is finite.  First, we observe
that, by completeness, we can find we can find unique paths $l,m$ on
$[0,s]$ extending this collection of compatible partial paths.  Moreover
$\MJWPhi(m)$ is also a path which extends the collection, so
$\MJWPhi(m)=l$. Similarly, $\MJWPsi(l)=m$, and so we have a compatible pair of
partial paths on $[0,s]$.

Secondly, we can repeat the above argument to find a pair that agree
on $[0,s+\eps]$ for some $\eps>0$, contradicting the definition of $s$.

Thus $s=\infty$, and it follows, by the earlier claim, that our
infinite paths $l$ and $m$ exist.  From the nature of this
construction it is clear that the paths $l$ and $m$ are unique.
\end{proof}

We now give an example of a game where this theorem is useful: note that
in this example the time that the player commits to doing something
does depend on the current position of the adversary (unlike in the
lion and man game).

The game is called `porter and student'. The game is played in the box
$[-1,1]^2$, with say the student starting at the origin $(0,0)$ and
the four porters at $(\pm1,0)$ and$(0,\pm1)$. The porters are
restricted to the boundary of the box: that is, points with (at least)
one coordinate $\pm1$. As usual, all players can run at
speed 1. The student's aim is to reach the boundary of the box
without being caught by a porter when he reaches it. If he gets
caught, or if he never reaches the boundary, then he loses.

There is an obvious strategy for the porters, namely: each porter
stays on the side of the box he starts on and keeps the other
coordinate the same as the student's. However, this is not locally
finite.

\begin{theorem}
  In the porter and student game the porters have a locally finite
  winning strategy. In particular, the student does not have a winning
  strategy.
\end{theorem}
We remark that the `in particular' is by the obvious analogue of
Proposition~\ref{p:lf} for this game.

\begin{proof}
  We give a locally finite winning strategy for the porters. We
  describe it for the porter on the left side; the others follow the
  same strategy rotated.

The porter aims to make sure that, at all times, his distance from the
top corner is at most the student's, and similarly for his distance
from the bottom corner.

If the student's distance from each corner is strictly less than the
porter's then the porter stays still for time 
\[
\min\{d(p,(0,1))-d(s,(0,1),
d(p,(0,-1))-d(s,(0,-1)\}.
\]
If the distance from the student to the top corner is the same as the
porter's then the porter runs at full speed towards the top corner. He
does this for time 
\[
\tfrac12(d(p,(0,-1))-d(s,(0,-1))>0.
\]
Similarly if the student's distance to the bottom corner is the same as
the porter's.

It is easy to check that this does give rise to a well-defined locally
finite strategy (for example, by considering the supremum of the times
up to which this gives a well-defined locally finite strategy), with
the porter satisfying his aim of making sure he stays at least as near
to each of his corners as the student. Hence if the student reaches
this edge he is caught by the porter.
\end{proof}

This gives an interesting contrast. Suppose that a porter is guarding
an infinite line from a student: so the porter lives on the $y$-axis,
and the student is in the left half-plane. The porter of course has a
winning strategy (keeping his $y$-coordinate the same as that of the
student), while the student also has a winning strategy: first achieve
a different $y$-coordinate to that of the porter, by using the `get
out from underneath' idea of Section 3, and then run straight to the
$y$-axis at a very small angle.  But if instead the porter is guarding
just a finite interval, then the student has no winning strategy
(because, as we have seen, the porter has a locally finite winning
strategy).


\section{Thanks}
We are extremely grateful to Hallard Croft for suggesting the problem in
the first place and for many interesting conversations. We would also like
to thank Robert Johnson for many interesting conversations.

\end{document}